\documentclass[12pt, a4paper]{amsart}

\usepackage{amsmath,amssymb, ulem}

\usepackage[pdftex, pdfstartview=FitH,colorlinks=true,linkcolor=blue,%
urlcolor=blue,citecolor=blue,pagebackref=true]{hyperref}

\usepackage[abbrev]{amsrefs}

\usepackage{xcomment}

\usepackage{graphicx}

\newtheorem {thm}    {Theorem}
\newtheorem {lem}      [thm]    {Lemma}

\newtheorem {cor}  [thm] {Corollary}
\newtheorem {prp}[thm]  {Proposition}

\theoremstyle{definition}

\newtheorem {dfn}      [thm]    {Definition}

\newcounter{AbcT}

\renewcommand{\a}{\alpha}
\renewcommand{\b}{\beta}
\renewcommand{\d}{\delta}

\newcommand{\e}{\varepsilon}

\renewcommand{\l}{\lambda}

\newcommand{\R}{{\bf R}}

\newcommand{\Q}{{\bf Q}}

\newcommand{\Z}{{\bf Z}}

\newcommand {\cL} {{\mathcal L}}
\newcommand {\cA} {{\mathcal A}}
\newcommand {\cB} {{\mathcal B}}

\newcommand {\cS} {{\mathcal S}}

\newcommand{\wt}{\widetilde}

\DeclareMathOperator{\ord}{ord}

\begin{document}
	
\title[Times $2$, $3$ mod $Q$.]{On the multiplicative group generated by two primes in $\Z/Q\Z$}
\author{P\'eter P. Varj\'u}
\address{P\'eter P. Varj\'u, Centre for Mathematical Sciences, Wilberforce Road, Cambridge CB3 0WB, UK}
\email{pv270@dpmms.cam.ac.uk}

\thanks{
The author has received funding from the European Research Council (ERC) under the European Union’s Horizon 2020 research and innovation programme (grant agreement No. 803711). The author was supported by the Royal Society.}

\dedicatory{Dedicated to the memory of Jean Bourgain.}

\begin{abstract}
We study the action of the multiplicative group generated by two prime numbers in $\Z/Q\Z$.
More specifically, we study returns to the set $([-Q^\e,Q^\e]\cap \Z)/Q\Z$.
This is intimately related to the problem of bounding the greatest common divisor of $S$-unit differences,
which we revisit.
Our main tool is the $S$-adic subspace theorem.
\end{abstract}

\maketitle

\section{Introduction}

In this note, we study the multiplicative group $\{p^mq^n:m,n\in\Z\}$ acting on $\Z/Q\Z$,
where $p$ and $q$ are prime numbers and $Q\in\Z_{\ge 2}$ with
$\gcd(Q,pq)=1$.
We are interested in returns to the set $([-Q^\b,Q^\b]\cap\Z)/Q\Z$ for some $\b\in(0,1)$.
That is, we aim to describe the set of $(m,n)\in\Z^2$ such that $p^mq^na=b$
for some $a,b\in ([-Q^\b,Q^\b]\cap\Z)/Q\Z$.
If $a$ and $b$ lifts to integers that are much smaller than $Q^\b$, then small perturbations
of $(m,n)$ will also satisfy the same property.
To eliminate this triviality, we restrict our attention to the following subset of
$([-Q^\b,Q^\b]\cap\Z)/Q\Z$.
\begin{dfn}
We write $B(\b,Q)$ for the set of residues $a\in\Z/Q\Z$ that have lifts
$\wt a$ in $[-Q^\b,Q^\b]\cap\Z$ with $\gcd(\wt a,pq)=1$.
\end{dfn}
Notice that $B(\b,Q)$ contains a canonical representative of each ``short
orbit segment'' intersecting $([-Q^\b,Q^\b]\cap\Z)/Q\Z$ in the following sense.
Given $a\in ([-Q^\b,Q^\b]\cap\Z)/Q\Z$, there 
is $a'\in B(\b,Q)$ and $m',n'\in \Z$ such that $a=p^{m'}q^{n'}a'$
and
\begin{align*}
|m'|\le& \b \log Q/\log p,\\ 
|n'|\le& \b \log Q/\log q.
\end{align*}
The choice of $a'$ is unique provided $\b<1/3$.
This means that, in a sense, to understand returns to $([-Q^\b,Q^\b]\cap\Z)/Q\Z$, it is enough
to understand returns to $B(\b,Q)$.

Our first main result is the following.

\begin{thm}\label{th:main}
Let $p$ and $q$ be two prime numbers, and let $K\in\Z_{\ge 1}$.
Then there is $C\in\R_{>1}$ and $\b\in\R_{>0}$ such that for all $Q\in\Z_{>C}$,
the set of $(m,n)\in\Z^2$ satisfying the conditions
\begin{itemize}
\item $|m|\le K\log Q/\log |p|$,
\item $|n|\le K\log Q/\log |q|$,
\item there are $a,b\in B(\b,Q)$ such that $p^mq^na = b$
\end{itemize}
is contained in a line.

The constant $C$ is ineffective, but $\b$ can be made explicit.
In particular, the theorem holds with $\b=(147K)^{-1}$
with some $C$ that is suitably large depending on $p$, $q$ and $K$.
\end{thm}

As can be seen from the proof, the result remains valid if we require only that $p$ and $q$ are multiplicatively independent
integers instead of being primes.
However, in that more general setting, it is less natural to restrict our study to the set $B(\b,Q)$.
Instead, one might formulate a result in terms of the set $([-Q^\b,Q^\b]\cap\Z)/Q\Z$ in place of $B(\b,Q)$
and replace the conclusion by saying that the resulting set of $(m,n)$ will be contained in a suitable neighborhood
of a line.
We leave this to the interested reader.

Theorem \ref{th:main} has the following corollary.

\begin{cor}\label{cr:order}
Let $p$ and $q$ be two prime numbers.
For an integer $Q\in\Z_{\ge 2}$ with $\gcd(pq,Q)=1$, we write
$\ord(Q)$ for the order of the multiplicative group generated by $p$ and $q$ in $\Z/Q\Z$.
Then
\[
\lim_{Q\to \infty}\frac{\ord(Q)}{(\log Q)^2}=\infty.
\]
\end{cor}

Again, this remains valid if we replace the condition of primality for $p$ and $q$ by multiplicative independence.

Corollary \ref{cr:order} is not a new result.
It is well known to follow from a result of Hern\'andez, Luca \cite{HL} and Corvaja, Zannier \cite{CZ}, which 
we will recall below.

Before that, we discuss how Corollary \ref{cr:order} follows from Theorem \ref{th:main}.
We observe that the set
\[
\Lambda=\{(m,n)\in\Z^2: p^mq^n\equiv 1 \mod Q\}
\]
is a sublatice of $\Z^2$ and its index is $\ord(Q)$.
We write $\l_1$ for the first and $\l_2$ for the second minima of $\Lambda$.
If $(m,n)\in\Lambda$ is non-zero, then necessarily 
\[
|m|\log p+|n|\log q\ge \log Q,
\]
so $\l_1\ge c\log Q$ for some constant $c$ that depends on $p$, $q$ and our choice for the norm
with respect to which the minima are defined.
By Theorem \ref{th:main}, $\l_2/\log Q\to \infty$ as $Q\to \infty$.
Corollary \ref{cr:order} now follows from Minkowski's theorem on successive minima.

Now we discuss some relevant results from the literature.
Bugeaud, Corvaja and Zannier \cite{BCZ}*{Theorem 1} proved that
\[
\gcd(a^n-1,b^n-1)\le \max(a^n,b^n)^{-\e}
\]
for all pair of multiplicatively
independent integers $a, b$ and for all $\e>0$ provided $n$ is sufficiently large depending on $a$, $b$ and $\e$.
This has been extended both by Hern\'andez, Luca \cite{HL}*{Theorem 2.1} and Corvaja, Zannier
\cite{CZ}*{Remark 1} to the case when $a^n$ and $b^n$
are replaced by two multiplicatively independent integers $u$ and $v$ containing prime factors only from a previously fixed set of primes $S$.
They proved that the inequality 
\[
\gcd(u-1,v-1)\le\max(|u|,|v|)^{-\e}
\]
holds provided $\max(|u|,|v|)$ is sufficiently large depending on
$S$ and $\e$.
This result is well known to imply Corollary \ref{cr:order}.
See also Corvaja, Rudnick, Zannier \cite{CRZ} for a related application of these methods
to the multiplicative order of integer matrices $\mod Q$, which contains Corollary
\ref{cr:order}.
A further extension was obtained by Luca \cite{Luc}*{Theorem 2.1}, who allows $u$ and $v$ to be rational numbers that may contain
prime factors outside $S$, provided their product (with multiplicities) is small compared to $\max(|u|,|v|)$.
Furthermore, in this work $u$ and $v$ are allowed to be multiplicatively dependent, provided they have no multiplicative relation
with small exponents.

Corvaja and Zannier \cite{CZ-height-rational-function} developed these ideas
in another direction to estimate the greatest common divisors of rational
functions evaluated at $S$ units.
These results have been extended by Levin \cite{Lev} to higher dimension.

See also the books of Zannier \cite{Zan} and Corvaja, Zannier \cite{CZ-book}, where
some of these results are discussed further.

We introduce some notation.
We fix a set $S$, which consists of a finite number of (finite) primes and the symbol $\infty$.
We write $S_f=S\backslash \{\infty\}$.
For each $v\in S$, we define a valuation $|\cdot|_v$ on $\Q$.
If $v$ is finite and $x\in\Z$, then we set $|x|_v=v^{-m}$, where $m$ is the largest integer
with $v^m|x$, and we extend $|\cdot|_v$ to $\Q$ multiplicatively.
This is the standard $v$-adic absolute value.
We define $|\cdot|_\infty$ to be the standard Archimedean absolute value.
We write $\cS$ for the set of positive integers all of whose prime factors are contained in $S_f$.

Now we can state our second main result, which extends the above mentioned result of Luca \cite{Luc}.

\begin{thm}\label{th:gcd}
For all $\e>0$ and $S$ as above, there are $C\in\R_{>1}$, $\a\in\R_{>0}$ and $N\in\Z_{>0}$
such that the following holds.

Let $a_1,b_1,a_2,b_2\in\Z$ be numbers that are not divisible by any prime in $S_f$.
Let $s_1,t_1,s_2,t_2\in\cS$.
Assume
\[
\gcd(a_1s_1,b_1t_1)=\gcd(a_2s_2,b_2t_2)=1.
\]
Let
\[
H=\max(s_1,t_1,s_2,t_2).
\]
Assume further that
\begin{equation}\label{eq:gcd}
\gcd(a_1s_1-b_1t_1,a_2s_2-b_2t_2)\ge H^\e.
\end{equation}

Then at least one of the following three items holds:
\begin{itemize}
\item[(a)] $H\le C$,
\item[(b)] $\max(a_1,b_1,a_2,b_2)\ge H^\a$,
\item[(c)] there are $n_1,n_2\in\Z$ not both $0$ such that $|n_1|,|n_2|\le N-1$ and
\[
\Big(\frac{a_1s_1}{b_1t_1}\Big)^{n_1}=\Big(\frac{a_2s_2}{b_2t_2}\Big)^{n_2}.
\]
\end{itemize}

The constant $C$ is ineffective, but $\a$ and $N$ can be made explicit.
The theorem always holds (with a suitably large $C$ depending on $\e$ and $S$) provided
\[
N=\Big\lfloor\frac{32}{7\e}\Big\rfloor, \qquad \a=\frac{7}{512}\e^2.
\]
\end{thm}

In fact, we will use in the proof only that $\e$, $N$ and $\a$ satisfy the inequalities
\begin{align}
(N+1)\e >& 2N^2\a+4,\label{eq:Cond1}\\
\e>&16(N-1)\a.\label{eq:Cond2}
\end{align}

This result improves on \cite{Luc}*{Theorem 2.1} in the following aspects.
\begin{itemize}
\item The result in \cite{Luc} is not applicable when $s_1,t_2,s_2,t_2$ are of comparable size.
\item The bound on $\max(a_1,b_1,a_2,b_2)$ in \cite{Luc} is of the form $H^{\a/\log\log H}$.
(Note that $H$ signifies a different quantity in the notation of \cite{Luc}, and we translated the bound to our notation.)
\item We make the value of $N$ explicit.
\end{itemize}

It was observed by Bugeaud, Corvaja and Zannier
that there are infinitely many values of $n$
such that
\[
\gcd(a^n-1,b^n-1)\ge\exp(\exp(c\log n/\log\log n)),
\]
where $a$ and $b$ are multiplicatively independent integers and $c>0$ is an absolute constant,
see the second remark after Theorem 1 in \cite{BCZ}.
This significantly limits the extent of any possible improvement over \eqref{eq:gcd}.
However, in this example, the greatest common divisor is highly composite, and it is not clear how large
a common prime factor of $s_1-1$ and $s_2-1$ can be for some $s_1,s_2\in\cS$.
This question is of particular interest in the context of Corollary \ref{cr:order} if we restrict $Q$ to be a prime.

It follows by the box principle that for any $Q\in\Z_{\ge 1}$ and for any $s\in\Z_{\ge 1}$, there are $a,b\in\Z$ with
$|a|,|b|\le Q^{1/2}$ such that $Q|as-b$.
This shows that we cannot hope to take $\a$ larger than $C\e$ in Theorem \ref{th:gcd} for some constant $C>0$.
However, this still leaves significant room for improvement.

Theorem \ref{th:main} easily follows from Theorem \ref{th:gcd}, which we show now.
\begin{proof}[Proof of Theorem \ref{th:main} assuming Theorem \ref{th:gcd}]
Suppose there are $a_1,b_1,a_2,b_2\in B(\b,Q)$ and
$(m_1,n_1),(m_2,n_2)\in\Z^2$ that are not collinear such that
\[
|m_1|,|m_2|\le\frac{K\log Q}{\log p},\qquad
|n_1|,|n_2|\le\frac{K\log Q}{\log q}
\]
and
\[
p^{m_1}q^{n_1}a_1= b_1,\qquad
p^{m_2}q^{n_2}a_2= b_2.
\]
We show that $Q$ must be bounded by a constant depending on $p$, $q$ and $K$ only.

To this end, we set $S=\{p,q,\infty\}$ and define $s_1,t_1,s_2,t_2\in\cS$ such that $s_1/t_1=p^{m_1}q^{n_1}$
and $s_2/t_2=p^{m_2}q^{n_2}$.
We denote by the same symbols the unique lifts of $a_1,b_1,a_2,b_2$ in $\Z\cap[-Q^{\b},Q^{\b}]$.
We assume without loss of generality that
\[
\gcd(a_1,b_1)=\gcd(a_2,b_2)=\gcd(s_1,t_1)=\gcd(s_2,t_2)=1.
\]
Since $a_1,b_1,a_2,b_2\in B(\b,Q)$, their lifts (denoted by the same symbol)
are not divisible by $p$ or $q$, so we get
\[
\gcd(a_1s_1,b_1t_1)=\gcd(a_2s_2,b_2t_2)=1.
\]
We also note that
\[
\gcd(a_1s_1-b_1t_1,a_2s_2-b_2t_2)\ge Q\ge H^{1/2K},
\]
where
\[
H=\max(s_1,t_1,s_2,t_2)\le \max(p^{|m_1|}q^{|n_1|},p^{|m_2|}q^{|n_2|}).
\]

Now we see that all the assumptions of Theorem \ref{th:gcd} hold with
\[
\e:=\log Q/\log H\ge 1/2K.
\]
Item (b) of the conclusion cannot hold, because
\[
\max(|a_1|,|b_1|,|a_2|,|b_2|)\le Q^\b= H^{\b\e},
\]
provided $\b$ is small enough so that $\b\e\le \a$.

Item (c) also cannot hold, because $(m_1,n_1)$ and $(m_2,n_2)$ are not collinear and this implies that
$a_1s_1/b_1t_1$ and $a_2s_2/b_2t_2$ are multiplicatively independent.
This means that item (a) must hold, which is precisely what we wanted to prove.

For this argument to work we only need that $\b$ is not larger than $\a/\e$.
With $\a=(7/512)\e^2$ and $\e\ge 1/2K$, we see that $\b=1/147K$ is sufficient.
\end{proof}

We prove Theorem \ref{th:gcd} in the next section.
The proof uses Schlickewei's $S$-adic generalization of Schmidt's subspace theorem.
The general approach goes back to the paper of Bugeaud, Corvaja and Zannier \cite{BCZ},
which has been developed further subsequently in
\cites{CRZ,HL,CZ,CZ-height-rational-function,Luc,Lev}.
Our proof makes use of the new construction introduced by Levin \cite{Lev} to choose the
linear forms for which the subspace theorem is applied.

\subsection{Notation}

Throughout the paper we fix a finite set $S$ that consists of some prime numbers
and the symbol $\infty$.
We write $S_f=S\backslash \{\infty\}$.
We write $\cS$ for the set of positive integers, all of whose prime divisors are in $S_f$.

When we have a notation similar to $X_1,\ldots,X_n$, we sometimes write $X_\bullet$ to refer to the whole
sequence, or to a generic element of the sequence.
The exact meaning will always be clear from the context.

The height of an integer vector $x\in\Z^d$ is defined as
\[
H(x_1,\ldots,x_d)=\max(|x_1|_\infty,\ldots,|x_d|_\infty),
\]
where $|\cdot|_\infty$ is the standard Archimedean absolute value on $\Q$.
\subsection*{Acknowledgments}

I am grateful to Elon Lindenstrauss for very helpful discussions on the subject of this note,
and to Yann Bugeaud, Pietro Corvaja, Umberto Zannier and the anonymous referee
for very helpful comments on an earlier
version of this note.

\section{Proof of Theorem \ref{th:gcd}}

The purpose of this section is the proof of Theorem \ref{th:gcd}.
Our main tool is Schmidt's subspace theorem in the following generalized form due to Schlickewei.

\begin{thm}[$S$-adic subspace theorem]
Let $d\in\Z_{\ge 2}$.
For each $v\in S$, let
$L_1^{(v)},\ldots,L_d^{(v)}\in\Q[x_1,\ldots,x_d]$ be linearly independent linear forms.
Then for all $\e>0$, the solutions $(x_1,\ldots,x_d)\in\Z^d$ of the inequality
\begin{equation}\label{eq:S-adic-forms}
\prod_{v\in S}\prod_{j=1}^{d}|L_j^{(v)}(x_1,\ldots,x_d)|_v\le H(x_1,\ldots,x_d)^{-\e}
\end{equation}
lie in a finite union of proper subspaces of $\Q^d$.
\end{thm}

See \cite{BG}*{Corollary 7.2.5} for a proof of this result.
In our applications, we will use the subspace theorem in a finite dimensional vector space $V$ over $\Q$, and to
facilitate the application of the subspace theorem, we need to fix an isomorphism from $V$ to $\Q^d$.
In these applications, there will be no natural choice for this isomorphism, and its exact choice will be largely immaterial.
For this reason, we reformulate the subspace theorem in the following equivalent form.

\begin{thm}\label{th:subspace}
Let $V$ be a $d\in\Z_{\ge 2}$ dimensional vector space over $\Q$.
For each $v\in S$, let $\Lambda_1^{(v)},\ldots,\Lambda_d^{(v)}$ be a basis
of the dual space $V^*$.
Furthermore, let $\Lambda_1^{(0)},\ldots,\Lambda_d^{(0)}$ be another basis of $V^*$.
Then for all $\e$, there is a finite set $\Phi_1,\ldots,\Phi_m\in V^*_{\neq 0}$ such that
every solution of
\[
\prod_{v\in S}\prod_{j=1}^d|\Lambda_j^{(v)}(x)|_v
\le H(\Lambda_1^{(0)}(x),\ldots,\Lambda_d^{(0)}(x))^{-\e}
\]
for $x\in V$ with $\Lambda^{(0)}_j(x)\in\Z$ for all $j=1,\ldots,d$
satisfies $\Phi_i(x)=0$ for some $i\in\{1,\ldots,m\}$.
\end{thm}

In our proof of Theorem \ref{th:gcd}, the first application of the subspace theorem will yield
a finite collection of polynomials in two variables depending only on $\e$ and $S$ such that one of the polynomials
must vanish at the point $(a_1s_1/b_1t_1,a_2s_2/b_2t_2)$ for any putative counterexample to the theorem.
After this, a second application of the subspace theorem will be needed to conclude the proof.
This second part of the proof amounts to proving the following statement.

\begin{prp}\label{pr:curve}
For all $\e>0$ and $S$ as above, there are $\a\in\R_{>0}$ and $N\in\Z_{>0}$
such that the following holds.
Fix a polynomial $P\in\Q[x_1,x_2]$ of degree at most $N-1$.
Then there is $C$ (depending on $P$, $S$ and $\e$) such that the following holds.

Let $a_1,b_1,a_2,b_2\in\Z$ be numbers that are not divisible by any prime in $S_f$.
Let $s_1,t_1,s_2,t_2\in\cS$.
Assume
\[
\gcd(a_1s_1,b_1t_1)=\gcd(a_2s_2,b_2t_2)=1
\]
and
\[
P\Big(\frac{a_1s_1}{b_1t_1},\frac{a_2s_2}{b_2t_2}\Big)=0.
\]
Let
\[
H=\max(s_1,t_1,s_2,t_2).
\]
Assume further that
\[
	\gcd(a_1s_1-b_1t_1,a_2s_2-b_2t_2)\ge H^\e.
\]

Then at least one of the following three items holds:
\begin{itemize}
	\item[(a)] $H\le C$,
	\item[(b)] $\max(a_1,b_1,a_2,b_2)\ge H^\a$,
	\item[(c)] there are $n_1,n_2\in\Z$ not both $0$ such that $|n_1|,|n_2|\le N-1$ and
	\[
	\Big(\frac{a_1s_1}{b_1t_1}\Big)^{n_1}=\Big(\frac{a_2s_2}{b_2t_2}\Big)^{n_2}.
	\]
\end{itemize}

The constant $C$ is ineffective, but $\a$ and $N$ can be made explicit.
The proposition always holds (with a suitably large $C$ depending on $\e$, $S$ and $P$) provided
\[
N=\Big\lfloor\frac{32}{7\e}\Big\rfloor, \qquad \a=\frac{7}{512}\e^2.
\]
\end{prp}

Notice that this is just a restatement of Theorem \ref{th:gcd} with the additional
assumption that the point $(a_1s_1/b_1t_1,a_2s_2/b_2t_2)$ is restricted to a curve.
This result is unlikely to be either new or optimal.
However it suffices for our purposes and the proof is simple, so we include it after we showed how
Theorem \ref{th:gcd} can be reduced to it.

The construction of the linear forms in the following proof is essentially a special
case of the construction of Levin \cite{Lev}*{Proof of Theorem 3.2}

\begin{proof}[Proof of Theorem \ref{th:gcd} assuming Proposition \ref{pr:curve}]
Let $\e\in\R_{>0}$, and let $\a\in\R_{>0}$ and $N\in \Z_{>0}$ satisfy
\eqref{eq:Cond1}--\eqref{eq:Cond2}.
We also fix some $a_1,b_1,a_2,b_2,s_1,t_1,s_2,t_2$ that satisfy all hypotheses
of Theorem \ref{th:gcd} and which fail items (b) and (c) of the conclusion.
We aim to show that item (a) of the conclusion holds, that is, $H\le C$ for some constant
$C$ depending only on $\e$ and $S$.

We let
\[
Q=\gcd(a_1s_1-b_1t_1,a_2s_2-b_2t_2).
\]
We assume, as we may that $Q$ is not divisible by any prime in $S_f$.
If we had $p|Q$ for some $p\in S_f$, then necessarily $p\nmid s_1t_1s_2t_2$, and we
could just omit $p$ from $S$.

In what follows, we consider the space $\Q^{N^2}\equiv\Q^{\{0,\ldots,N-1\}^2}$,
and write
\[
y=(y_{l_1,l_2})_{l_1=0,\ldots,N-1,l_2=0,\ldots,N-1}
\]
for its typical element.
We will apply the subspace theorem for the quotient space
\[
V=\Q^{\{0,\ldots,N-1\}^2}/\{(z,z,\ldots,z)\}.
\]

We will evaluate the linear forms at the point $\wt y\in V$ whose coordinates are
\[
\wt y_{l_1,l_2}
=\frac{a_1^{l_1}s_1^{l_1}b_1^{N-1-l_1}t_1^{N-1-l_1}a_2^{l_2}s_2^{l_2}b_2^{N-1-l_2}t_2^{N-1-l_2}}{Q}.
\]
Strictly speaking, this specifies a point in $\Q^{\{0,\ldots,N-1\}^2}$, but we do not
distinguish $\wt y$ from its projection to $V$ in our notation.

For each $v\in S$, let $(l_1^{(v)},l_2^{(v)})$ be such that
$|\wt y_{l_1,l_2}|_v$ is minimal for $(l_1,l_2)=(l_1^{(v)},l_2^{(v)})$.
We define the set of linear forms $\Lambda^{(v)}_\bullet\in V^{*}$ to be an enumeration of
the forms
\[
y\mapsto y_{l_1,l_2}-y_{l_1^{(v)},l_2^{(v)}}
\]
for $(l_1,l_2)\in\{0,\ldots,N-1\}^2\backslash(l_1^{(v)},l_2^{(v)})$.
It is easy to verify that these are indeed in $V^*$, that
is they are constant on cosets of the line $\{(z,z,\ldots,z)\}$, and that
they also form a basis.

We also define $\Lambda^{(0)}_\bullet=\Lambda^{(\infty)}_\bullet$, say.
We note that
\[
\wt y_{l_1,l_2}-\wt y_{l_1',l_2'}\in\Z
\]
for all $l_1,l_2,l_1',l_2'\in\{0,\ldots,N-1\}$, since
\[
a_1s_1\equiv b_1t_1,\quad a_2s_2\equiv b_2t_2\mod Q,
\]
and hence
\[
a_1^{l_1}s_1^{l_1}b_1^{N-1-l_1}t_1^{N-1-l_1}a_2^{l_2}s_2^{l_2}b_2^{N-1-l_2}t_2^{N-1-l_2}\mod Q
\]
is independent of $l_1$ and $l_2$.
For this reason
\[
\Lambda^{(0)}_\bullet(\wt y)\in\Z^{N^2-1}.
\]

We observe that
\[
|\wt y_{l_1,l_2}-\wt y_{l_1^{(v)},l_2^{(v)}}|_v\le C_v |\wt y_{l_1,l_2}|_v
\]
for each $(l_1,l_2)$ and $v\in S$, where $C_v=1$ if $v$ is finite and $C_\infty=2$.
This means that we have
\begin{equation}\label{eq:fraction}
\prod_{v\in S}{\prod}^\bullet|\Lambda_\bullet^{(v)}(\wt y)|_v
\le2^{N^2-1}\frac{\prod_{v\in S}\prod_{l_1=0}^{N-1}\prod_{l_2=0}^{N-1}|\wt y_{l_1,l_2}|_v}
{\prod_{v\in S}|\wt y_{l_1^{(v)},l_2^{(v)}}|_v}.
\end{equation}
Here $\prod^\bullet$ signifies multiplication over the index suppressed by the $\bullet$ notation. 

We first estimate the numerator in \eqref{eq:fraction}.
For each $(l_1,l_2)\in\{0,\ldots,N-1\}^2$, we have
\[
\prod_{v\in S}|\wt y_{l_1,l_2}|_v
=\frac{|a_1^{l_1}b_1^{N-1-l_1}a_2^{l_2}b_2^{N-1-l_2}|_\infty}{|Q|_\infty}
\le H^{2(N-1)\a}Q^{-1}.
\]
This gives us
\[
\prod_{v\in S}\prod_{l_1=0}^{N-1}\prod_{l_2=0}^{N-1}|\wt y_{l_1,l_2}|_v
\le H^{2(N-1)N^2\a}Q^{-N^2}.
\]

Next, we estimate the denominator in \eqref{eq:fraction}.
We note that
\[
|\wt y_{l_1^{(v)},l_2^{(v)}}|_\infty\ge Q^{-1}.
\]
Furthermore, we have
\[
|\wt y_{l_1^{(v)},l_2^{(v)}}|_v\ge |s_1^{N-1}t_1^{N-1}s_2^{N-1}t_2^{N-1}|_v
\]
for all finite $v\in S$, hence
\[
\prod_{v\in S_f}|\wt y_{l_1^{(v)},l_2^{(v)}}|_v\ge s_1^{-N+1}t_1^{-N+1}s_2^{-N+1}t_2^{-N+1}
\ge H^{-4(N-1)}.
\]
(Here we used that $Q$ is not divisible by any prime in $S$.)

Combining our estimates for the numerator and denominator in \eqref{eq:fraction},
we get
\begin{align*}
\prod_{v\in S}{\prod}^\bullet|\Lambda_\bullet^{(v)}(\wt y)|_v
\le&2^{N^2-1} Q^{-N^2+1} H^{2(N-1)N^2\a+4(N-1)}\\
\le& 2^{N^2-1} H^{2(N-1)N^2\a+4(N-1)-(N^2-1)\e}.
\end{align*}
We write
\[
2\d=(N^2-1)\e-2(N-1)N^2\a-4(N-1),
\]
which is positive by \eqref{eq:Cond1}.
We assume as we may that $2^{N^2-1}\le H^\d$, for otherwise
$H\le 2^{\d^{-1}(N^2-1)}$
and we see that item (a) of the conclusion holds.
Therefore, we have
\[
\prod_{v\in S}{\prod}^{\bullet}|\Lambda_\bullet^{(v)}(\wt y)|_v\le H^{-\d}.
\]

We observe that
\[
H(\Lambda_\bullet^{(0)}(\wt y))
\le 2 \max_{l_1,l_2} |\wt y_{l_1,l_2}|_\infty
\le H^{2(1+\a)(N-1)},
\]
and hence
\[
\prod_{v\in S}{\prod}^\bullet|\Lambda_\bullet^{(v)}(\wt y)|_v
\le H(\Lambda_\bullet^{(0)}(\wt y))^{-\frac{\d}{2(1+\a)(N-1)}}.
\]

This means that the subspace theorem applies and we conclude that there is
a finite collection of linear forms $\Phi_\bullet$ such that $\Phi_j(\wt y)=0$
for some $j$.
It may appear that the set of linear forms $\Phi_\bullet$ depends on $\wt y$,
for the linear forms $\Lambda_\bullet^{(v)}$ were chosen in a manner depending on
it.
However, there are only finitely many possibilities, and if we take $\Phi_\bullet$
to be the union of all linear forms that we obtain from each possible application
of the subspace theorem, then it is independent of $\wt y$.

Now $\Phi_j$ lifts to a nonzero linear form on $\Q^{N\times N}$, and it
induces a non-zero polynomial $P_j\in\Q[x_1,x_2]$ such that
\[
P_j\Big(\frac{a_1s_1}{b_1t_1},\frac{a_2s_2}{b_2t_2}\Big)=0.
\]
We can now apply Proposition \ref{pr:curve} for each polynomial $P_j$ that arises
in this way and we conclude the proof.
\end{proof}

We turn to the proof of Proposition \ref{pr:curve}.
It requires the following simple lemma.

\begin{lem}\label{lm:diff}
Let $y_1\neq y_2\in\Z$, $Q\in\Z_{\neq 0}$ be such that $Q|y_1-y_2$ and $Q$ is not divisible by any primes in $S$.
Then
\[
\prod_{v\in S} \min(|y_1|_v,|y_2|_v)\le \frac{2}{Q} \cdot \prod_{v\in S} |y_1y_2|_v
\]
\end{lem}

\begin{proof}
It is clear that neither the assumptions nor the conclusion of the lemma changes if we
divide both $y_1$ and $y_2$ by a divisor of $\gcd(y_1,y_2)$ all of whose prime factors are
in $S$.
For this reason we may assume that $\gcd(y_1,y_2)$ contains no prime factor that is in $S$.

We have
\[
\prod_{v\in S} \min(|y_1|_v,|y_2|_v)= \prod_{v\in S} \frac{|y_1y_2|_v}{\max(|y_1|_v,|y_2|_v)}.
\]
Since $\gcd(y_1,y_2)$ contains no prime factor that is in $S$, we have $\max(|y_1|_v,|y_2|_v)=1$ for all finite places $v\in S$.
In addition, we have $\max(|y_1|_\infty,|y_2|_\infty)\ge Q/2$, because $y_1$ and $y_2$ are distinct integers whose difference
is divisible by $Q$.
Plugging these observations into the above identity, we get the claim of the lemma.
\end{proof}

\begin{proof}[Proof of Proposition \ref{pr:curve}]
Let $\e\in\R_{>0}$, and let $\a\in\R_{>0}$ and $N\in \Z_{>0}$ satisfy
\eqref{eq:Cond1}--\eqref{eq:Cond2}.
Let $P$ be as in Proposition \ref{pr:curve}.
We also fix some $a_1,b_1,a_2,b_2,s_1,t_1,s_2,t_2$ that satisfy all hypotheses
of the proposition and which fail items (b) and (c) of the conclusion.
We aim to show that item (a) of the conclusion holds.

We assume without loss of generality that $P$ is irreducible.
Write $d_1$ and $d_2$ for the degrees of $P$ in $x_1$ and $x_2$ respectively, and let
\[
P(x_1,x_2)=\sum_{j_1=0}^{d_1}\sum_{j_2=0}^{d_2} \a_{j_1,j_2} x_1^{j_1}x_2^{j_2}.
\]
We note that $d_1,d_2\le N-1$ by assumption.

We also assume without loss of generality that $d_1,d_2\ge 1$.
Indeed, if we had $d_2=0$, say, then there would be only finitely many possibilities
for $a_1s_1/b_1t_1$ such that $P(a_1s_1/b_1t_1,\cdot)=0$ holds, and this in turn
restricts $a_1$, $s_1$, $b_1$, $t_1$ to a finite set.
This imposes an upper bound on $\gcd(a_1s_1-b_1t_1,a_2s_2-b_2t_2)$ and hence on $H$
unless $a_1s_1-b_1t_1=0$.
However, this latter case is not possible, because item (c) of the conclusion would
hold with
\[
\Big(\frac{a_1s_1}{b_1t_1}\Big)^{1}=\Big(\frac{a_2s_2}{b_2t_2}\Big)^{0}.
\]
We see that $d_1=0$ or $d_2=0$ implies that item (c) of the conclusion holds, so we can indeed
assume $d_1,d_2\ge 1$

We also note that at least one in each of the four sets $\a_{0,\bullet}$, $\a_{d_1,\bullet}$,
$\a_{\bullet,0}$ and $\a_{\bullet,d_2}$ of coefficients does not vanish.
(Here we used that $P$ is irreducible and $P\neq x_1$ and $P\neq x_2$.)

In what follows, we consider the space $\Q^{2d_1\times 2d_2}$, whose
typical element is denoted by
\[
y=(y_{l_1,l_2})_{l_1=0,\ldots,2d_1-1, l_2=0,\ldots, 2d_2-1}.
\]
For $m_1=0,\ldots,d_1-1$ and $m_2=0,\ldots,d_2-1$, we write
\[
\Psi_{m_1,m_2}(y)=\sum_{j_1=0}^{d_1}\sum_{j_2=0}^{d_2} \a_{j_1,j_2} y_{j_1+m_1,j_2+m_2},
\]
which is a linear form on $\Q^{2d_1\times 2d_2}$.
We observe that a point $(x_1,x_2)\in\Q^2_{\neq 0}$ satisfies $P(x_1,x_2)$ if and only if
\[
\Psi_{m_1,m_2}((x_1^{l_1}x_2^{l_2})_{l_1=0,\ldots,2d_1-1, l_2=0,\ldots, 2d_2-1})=0
\]
holds for at least one, and hence for all $m_1,m_2$ in the relevant range.
We write $V$ for the $3d_1d_2$ dimensional subspace of $\Q^{2d_1\times 2d_2}$
on which all $\Psi_{m_1,m_2}$ vanish.

We consider the point $\wt y\in V$ given by
\[
\wt y_{l_1,l_2}=a_1^{l_1}s_1^{l_1}b_1^{2d_1-1-l_1}t_1^{2d_1-1-l_1}
a_2^{l_2}s_2^{l_2}b_2^{2d_2-1-l_2}t_2^{2d_2-1-l_2}.
\]
To verify that $\Psi_{m_1,m_2}(\wt y)=0$, we note that
\[
b_1^{-2d_1+1}t_1^{-2d_1+1}b_2^{-2d_2+1}t_2^{-2d_2+1}\cdot\wt y_{l_1,l_2}
=\Big(\frac{a_1s_1}{b_1t_1}\Big)^{l_1}\Big(\frac{a_2s_2}{b_2t_2}\Big)^{l_2}.
\]

In what follows, we use the subspace theorem to show that there is a finite collection
$\Phi_\bullet\in V^*_{\neq 0}$ such that $\Phi_j(\wt y)=0$ for some $j$ and
this collection of linear forms is independent of the choice of $a_1,b_1,a_2,b_2,s_1,t_1,s_2,t_2$.
Each $\Phi_j$ can be lifted to a linear form on $\Q^{2d_1\times 2d_2}$, which
is not in the span of the $\Psi_{m_1,m_2}$.
We denote this linear form with the same symbol.
Then the polynomial
\[
Q_j(x_1,x_2)=\Phi_j((x_1^{l_1}x_2^{l_2})_{l_1=0,\ldots,2d_1-1,l_2=0,\ldots,2d_2-1})
\]
is not in the ideal generated by $P$, but
\[
Q_j\Big(\frac{a_1s_1}{b_1t_1},\frac{a_2s_2}{b_2t_2}\Big)=0.
\]
Each such $Q_j$ has only finitely many common solutions with $P$.
This means that the point
\[
\Big(\frac{a_1s_1}{b_1t_1},\frac{a_2s_2}{b_2t_2}\Big)
\]
must belong to a certain finite set, which depends only on $P$ and $S$, and this
means that item (a) in the conclusion holds with some
$C$ that depends only on $P$ and $S$.
This will complete the proof.

The next step is to choose the families of linear forms on $V$ needed for
the application of the subspace theorem.
For each place $v\in S$, we choose a set $\cL_v\subset\{0,\ldots,2d_1-1\}\times\{0,\ldots,2d_2-1\}$
of cardinality $\dim V=3d_1d_2$.
We then define $\Lambda^{(v)}_\bullet$
to be an enumeration of the linear forms
$y\mapsto y_{l_1,l_2}$ for $(l_1,l_2)\in\cL_v$.

Let $i$ be the smallest and let $k$ be the largest index such that $\a_{0,i}\neq0$ and $\a_{d_1,k}\neq0$, respectively.
(Recall that $\a_{j_1,j_2}$ are the coefficients of $P$.)
Each of the sets $\cL_v$ will be either
\begin{equation}\label{eq:form1}
\{0,\ldots,2d_1-1\}\times\{0,\ldots,2d_2-1\}\backslash\{d_1,\ldots,2d_1-1\}\times\{k,\ldots,k+d_2-1\}
\end{equation}
or
\begin{equation}\label{eq:form2}
\{0,\ldots,2d_1-1\}\times\{0,\ldots,2d_2-1\}\backslash\{0,\ldots,d_1-1\}\times\{i,\ldots,i+d_2-1\}.
\end{equation}

We first show that the resulting linear forms $\Lambda_\bullet^{(v)}$ form a basis of $V^*$ in either case.
In fact, we show this only in the case of \eqref{eq:form1}, because the case of \eqref{eq:form2}
can be treated in a similar fashion.
Since $|\eqref{eq:form1}|=\dim V$, it is enough to show that the linear forms $y\mapsto y_{l_1,l_2}$
for $(l_1,l_2)\in\eqref{eq:form1}$ span $V^*$.
To that end, it is enough to show that $y\mapsto y_{l_1,l_2}$ is in the span for all
$(l_1,l_2)\in\{d_1,\ldots,2d_1-1\}\times\{k,\ldots,k+d_1-1\}$.
Fix some $(l_1',l_2')\in\{d_1,\ldots,2d_1-1\}\times\{k,\ldots,k+d_1-1\}$.
We observe that
\[
y_{l_1',l_2'}=-\sum_{{(j_1,j_2)\neq(d_1,k)}}
\frac{\a_{j_1,j_2}}{\a_{d_1,k}} y_{j_1+l_1'-d_1,j_2+l_2'-k}
\]
for all $y\in V$.
This means that $y\mapsto y_{l_1',l_2'}$ is in the span of the linear forms $y\mapsto y_{l_1,l_2}$
for
\[
(l_1,l_2)\in\{0,\ldots,l_1'-1\}\times\{0,\ldots,2d_2-1\}\cup\{l_1'\}\times\{0,\ldots,l_2'-1\}.
\]
Using this observation, we can prove that $(l_1',l_2')$ is in the span of $y\mapsto y_{l_1,l_2}$
for $(l_1,l_2)\in\eqref{eq:form1}$ by induction first on $l_1'$ and then on $l_2'$.

For each $v\in S$, we define $\Lambda_{\bullet}^{(v)}$ using $\eqref{eq:form1}$ if
$|\wt y_{d_1,k}|_v\ge|\wt y_{0,i}|_v$ and we use $\eqref{eq:form2}$ otherwise.
We write $\cA=\eqref{eq:form1}\cap\eqref{eq:form2}$ and $\cB=\eqref{eq:form1}\backslash\eqref{eq:form2}$.
We observe that $\{0,\ldots,2d_1-1\}\times\{0,\ldots,2d_2-1\}$ is the disjoint union of the sets
$\cA$, $\cB$ and $\cB+(d_1,k-i)$.
For each $v\in S$, $\Lambda_\bullet^{(v)}$ contains $y\mapsto y_{l_1,l_2}$ for all $(l_1,l_2)\in\cA$ and
it also contains precisely one of $y\mapsto y_{l_1,l_2}$ or $y\mapsto y_{l_1+d_1,l_2+k-i}$ for each $(l_1,l_2)\in\cB$,
and it contains the one which gives a smaller or equal $|\cdot|_v$ value to $\wt y$.
This means that
\begin{align*}
\prod_{v\in S}{\prod}^{\bullet}|\Lambda_\bullet^{(v)}(\wt y)|_v=&\prod_{v\in S}\prod_{(l_1,l_2)\in \cA} |\wt y_{l_1,l_2}|_v\\
&\times\prod_{v\in S}\prod_{(l_1,l_2)\in\cB}\min(|\wt y_{l_1,l_2}|_v,|\wt y_{l_1+d_1,l_2+k-i}|_v).
\end{align*}
Here $\prod^\bullet$ signifies multiplication over the index suppressed by the $\bullet$ notation. 

We note that $\wt y_{l_1,l_2}\neq \wt y_{l_1+d_1,l_2+k-i}$ for each $(l_1,l_2)\in\cB$
follows from
\[
\Big(\frac{a_1s_1}{b_1t_1}\Big)^{d_1}\neq\Big(\frac{a_2s_2}{b_2t_2}\Big)^{i-k},
\]
which in turn follows from our assumption that item (c) in the conclusion
does not hold.
Therefore, we can apply Lemma \ref{lm:diff} for each pair $\wt y_{l_1,l_2},\wt y_{l_1+d_1,l_2+k-i}$
for $(l_1,l_2)\in\cB$, and get
\[
\prod_{v\in S}{\prod}^{\bullet}|\Lambda_\bullet^{(v)}(\wt y)|_v
\le\Big(\frac{2}{Q}\Big)^{|\cB|}\cdot\prod_{v\in S}\prod_{l_1=0}^{2d_1-1}\prod_{l_2=0}^{2d_2-1} |\wt y_{l_1,l_2}|_v.
\]

We note that
\[
\prod_{v\in S}|\wt y_{l_1,l_2}|_v=|a_1^{l_1}b_1^{2d_1-1-l_1}a_2^{l_2}b_2^{2d_2-1-l_2}|_\infty\le H^{(2d_1+2d_2-2)\a}.
\]
This and $Q\ge H^{\e}$ gives
\[
\prod_{v\in S}{\prod}^{\bullet}|\Lambda_\bullet^{(v)}(\wt y)|_v\le 2^{d_1d_2}H^{4 d_1d_2(2d_1+2d_2-2)\a-d_1d_2\e}.
\]

We write
\[
2\d=\e-8(d_1+d_2-1)\a,
\]
which is positive by assumption \eqref{eq:Cond2}.
We assume as we may that $2\le H^\d$, for otherwise $H\le 2^{1/\d}$,
and item (a)
of the conclusion holds.
We have therefore,
\[
\prod_{v\in S}{\prod}^{\bullet}|\Lambda_\bullet^{(v)}(\wt y)|_v\le H^{-\d d_1d_2}.
\]

We apply the subspace theorem with the linear forms $\Lambda_\bullet^{(v)}$ defined above and with
$\Lambda^{(0)}_\bullet=\Lambda_\bullet^{(\infty)}$, say.
We note that
$\Lambda^{(0)}_\bullet(\wt y)\in\Z^{3d_1d_2}$
and
\[
H(\Lambda^{(0)}_\bullet(\wt y))\le\max_{l_1,l_2}|\wt y_{l_1,l_2}|_\infty\le H^{(2d_1+2d_2-2)(1+\a)}.
\]
We have therefore
\[
\prod_{v\in S}{\prod}^{\bullet}|\Lambda_\bullet^{(v)}(\wt y)|_v
\le H(\Lambda^{(0)}_\bullet(\wt y))^{-\frac{\d d_1d_2}{(2d_1+2d_2-2)(1+\a)}}.
\]
This means that the subspace theorem applies and hence there is a finite collection of linear forms
$\Phi_\bullet\in V^*_{\neq 0}$ such that $\Phi_j(\wt y)=0$ for some $j$.

It may appear that the linear forms $\Phi_\bullet$ depend on $\wt y$, because the choice of $\Lambda^{(v)}_\bullet$ for each $v\in S$ depends
on it.
However, there are only finitely many possibilities we need to consider, so we can simply take the union
of the linear forms that result from each possible application of the subspace theorem.
As we discussed above, this completes the proof.
\end{proof}

\bibliography{bib}

\end{document}